\documentclass[12pt, amsfonts]{amsart}

\usepackage{amsmath,amssymb,amsthm,color,enumerate,comment}
\usepackage[all]{xy}
\usepackage{tikz}
\usepackage{todonotes}
\usepackage{url} 
\usepackage{hyperref} 

\usepackage[top=35mm]{geometry} 

\numberwithin{equation}{section}

\SelectTips{eu}{12}

\newtheorem{thm}{Theorem}[section]

\newtheorem{theorem}[thm]{Theorem}
\newtheorem{proposition}[thm]{Proposition}
\newtheorem{lemma}[thm]{Lemma}
\newtheorem{corollary}[thm]{Corollary}

\theoremstyle{definition}

\newtheorem{definition}[thm]{Definition}
\newtheorem{example}[thm]{Example}

\theoremstyle{remark}

\newtheorem{remark}[thm]{Remark}

\newcommand{\RR}{\mathbb{R}}
\newcommand{\ZZ}{\mathbb{Z}}

\newcommand{\lN}{\overleftarrow{N}}
\newcommand{\rN}{\overrightarrow{N}}

\newcommand{\tPhi}{\tilde{\Phi}}
\newcommand{\hPsi}{\hat{\Psi}}

\newcommand{\hgamma}{\hat{\gamma}}

\newcommand{\XXX}{\mathcal{X}}
\newcommand{\III}{\mathcal{I}}
\newcommand{\NNN}{\mathcal{N}}
\newcommand{\DDD}{\mathcal{D}}
\newcommand{\EEE}{\mathcal{E}}

\newcommand{\CCC}{\check{\mathrm{C}}}

\newcommand{\tX}{\tilde{X}}

\newcommand{\rNNN}{\overrightarrow{\mathcal{N}}}
\newcommand{\lNNN}{\overleftarrow{\mathcal{N}}}

\newcommand{\JJJ}{\mathcal{J}}

\newcommand{\conv}{\mathrm{conv}}
\newcommand{\link}{\mathrm{link}}
\newcommand{\Bor}{\mathrm{Bor}}

\title[Closed neighborhood complex]{Closed neighborhood complexes of graphs}

\author{Takahiro Matsushita}
\address{Department of Mathematical Sciences, Faculty of Science, Shinshu University, Matsumoto, Nagano 390-8621, Japan}
\email{matsushita@shinshu-u.ac.jp}

\subjclass[2020]{Primary 05C69; Secondary 05C20, 55U10}

\keywords{neighborhood complex; independence complex; canonical double covering; neighborhood hypergraph}

\begin{document}


\maketitle

\begin{abstract}
The closed neighborhood complex $\mathcal{N}[G]$ of a simple graph $G$ is the simplicial complex whose simplices are finite sets of vertices contained in a closed neighborhood of a vertex in $G$. We reveal that the closed neighborhood complex has close connections with other concepts, including the independence complex of the canonical double covering and the independence complex of the neighborhood hypergraph. Furthermore, we show that the fundamental group of the closed neighborhood complex is isomorphic to Grigor'yan--Lin--Muranov--Yau's fundamental group of a graph introduced in the study of path homology.
\end{abstract}

\section{Introduction} \label{section introduction}

\subsection{Background and main object of this paper}

The neighborhood complex $\NNN(G)$ of a graph $G$ was introduced by Lov\'asz in his celebrated proof of the Kneser conjecture \cite{Lovasz}. Lov\'asz showed that there is a close relationship between the chromatic number $\chi(G)$ of a graph $G$ and the connectivity of $\NNN(G)$, and determined the chromatic numbers of the Kneser graphs. After that, the neighborhood complex and its generalizations have been studied by many authors (see \cite{BK1, BK2, DM, MatsushitaJMSUT, MW, Zivaljevic} for example).

To introduce the main object in this paper, we first recall the definition of the neighborhood complex. Let $G$ be a simple graph, and $v$ a vertex of $G$. The \emph{open neighborhood $N_G(v)$ of $v$} is the set of vertices adjacent to $v$. Then, the neighborhood complex $\NNN(G)$ is the abstract simplicial complex defined as follows: The ground set of $\NNN(G)$ is the vertex set $V(G)$ of $G$, and a finite subset $\sigma$ of $V(G)$ is a simplex of $\NNN(G)$ if and only if $\sigma$ is contained in the open neighborhood $N_G(v)$ of some vertex $v$ of $G$.

Recall that the \emph{closed neighborhood $N_G[v]$ of $v$ in $G$} is defined as $N_G(v) \cup \{ v\}$. Then, it is natural to consider what happens if a closed neighborhood is employed instead of an open neighborhood in the definition of the neighborhood complex. This natural consideration yields the object we mainly deal with in this paper:

\begin{definition}
The \emph{closed neighborhood complex $\NNN[G]$ of a simple graph $G$} is the abstract simplicial complex defined as follows: The ground set of $\NNN[G]$ is the vertex set $V(G)$ of $G$, and the set of simplices is
\[ \NNN[G] = \{ \sigma \subset V(G) \mid \textrm{$\# \sigma < \infty$ and there is $v \in V(G)$ such that $\sigma \subset N_G[v]$}\}.\]
\end{definition}

The \emph{reflexive closure $G^\circ$ of a simple graph $G$} is the reflexive graph obtained by adding a loop to each vertex. Then, the closed neighborhood complex $\NNN[G]$ of $G$ is the neighborhood complex $\NNN(G^\circ)$ of the reflexive closure of $G$. The neighborhood complex in fact was studied for graphs with loops (see \cite{BK1, DochtermannEJC, MatsushitaJMSUT} for example). Hence, the closed neighborhood complex introduced above is not new to the literature, and in fact there is some work (see \cite{ASS} for example) concerning the closed neighborhood complexes of specific graphs. However, since the neighborhood complex has primarily been studied in the context of the graph coloring problem, the neighborhood complex for graphs with loops has generally not been treated as a central topic.

On the other hand, in this paper we deal with the closed neighborhood complex. One reason is that the closed neighborhood complex has close connections with various concepts in combinatorics and topology. Furthermore, the closed neighborhood complex includes important examples such as \v{C}ech complexes, which are a fundamental object in applied topology (see \cite{AA, Carlsson, EH}).

We now recall the definition of \v{C}ech complex of a metric space $(X,d)$. Let $r$ be a non-negative real number. Then, the \emph{\v{C}ech complexes $\CCC(X,r)$ and $\CCC[X,r]$} are the abstract simplicial complexes whose vertex set is $X$ and whose set of simplices is
\[ \CCC(X,r) = \{ \sigma \subset X \mid \text{$\# \sigma < \infty$ and there is $a \in X$ such that $\sigma \subset B(a, r)$}\},\]
\[ \CCC[X,r] = \{ \sigma \subset X \mid \text{$\# \sigma < \infty$ and there is $a \in X$ such that $\sigma \subset B[a, r]$}\}.\]
Here, $B(a,r)$ (or $B[a,r]$) denotes the open (or closed, respectively) ball of radius $r$ centered at $a$.

The \v{C}ech complexes are realized as the closed neighborhood complexes of the graphs $G(X, r)$ and $G[X,r]$. The \emph{neighborhood graphs $G(X, r)$ and $G[X,r]$ of $(X,d)$} is the graph defined as follows: The vertex set of $G(X,r)$ and $G[X,r]$ are $X$, and
\[ E(G(X,r)) = \{ (x,y) \in X^2 \mid 0 < d(x,y) < r \},\]
\[ E(G[X,r]) = \{ (x,y) \in X^2\mid 0 < d(x,y) \le r\}.\]
Then, the closed neighborhood complexes $\NNN[G(X,r)]$ and $\NNN[G[X,r]]$ are the \v{C}ech complexes $\CCC(X,r)$ and $\CCC[X,r]$ of $X$, respectively.

In the study of the \v{C}ech complexes $\CCC(X,r)$ and $\CCC[X,r]$, $X$ is often not assumed to be finite (see \cite{AA} for example). Thus, unless otherwise stated, our graphs are undirected, simple and not necessarily finite.

\subsection{Connections with other concepts}

The closed neighborhood complex has close relationships with the independence complex of the canonical double covering and the independence complex of the neighborhood hypergraph. All necessary definitions will be given in subsequent sections.

Let $G = (V(G), E(G))$ be a graph. Recall that a subset $\sigma$ of $V(G)$ is \emph{independent in $G$} if no two elements of $\sigma$ are adjacent in $G$. The \emph{independence complex $\III(G)$ of $G$} is the simplicial complex whose ground set is the vertex set $V(G)$ of $G$ and whose set of simplices is the set of finite independent subsets of $G$. The independence complex of a simple graph is one of the central topics in topological combinatorics, and has been studied by many authors (see \cite{AdamaszekJCTA, BC, Bravo2025, EH2, GSS, KozlovJCTA, MWTA, MWHHA} for example).

Let $K_2$ denote the complete graph consisting of two vertices. Then, the categorical product $K_2 \times G$ is called the \emph{canonical double covering of $G$}. The following theorem reveals a close relationship between the closed neighborhood complex $\NNN[\overline{G}]$ of the complement of $G$ and the independence complex $\III(K_2 \times G)$ of the canonical double covering of $G$:

\begin{theorem} \label{theorem A}
Let $G$ be a graph. Then, we have the homotopy equivalence:
\[ \Sigma \NNN[\overline{G}] \simeq \III(K_2 \times G).\]
In particular, if $G$ is bipartite, then $\Sigma \NNN[\overline{G}] \simeq \III(G) * \III(G)$.
\end{theorem}

The following corollary is a restatement of the latter assertion of Theorem~\ref{theorem A}. Here, $\XXX(G)$ denotes the clique complex of a graph $G$ (see Section~\ref{section preliminaries} for the definition).

\begin{corollary} \label{corollary A}
Let $G$ be a graph such that $\overline{G}$ is bipartite, i.e., there are cliques $C_1$ and $C_2$ of $G$ such that $V(G) = C_1 \sqcup C_2$. Then, we have the following homotopy equivalence:
\[ \Sigma \NNN[G] \simeq \XXX(G) * \XXX(G).\]
\end{corollary}

By Theorem~\ref{theorem A}, one can determine the homotopy type of the independence complex $\III(K_2 \times G)$ of the canonical double covering of $G$ by determining the homotopy type of $\NNN[\overline{G}]$. We will determine the homotopy type of the independence complexes of the canonical double coverings of Borsuk graphs $\Bor(S^1; a)$ and $\Bor[S^1 ; a]$ of a circle, utilizing the homotopy types of \v{C}ech complexes of a circle determined by Adamaszek--Adams \cite{AA} (Example~\ref{example Borsuk graph}). On the other hand, one can sometimes determine the homotopy type of $\NNN[G]$ by the homotopy type of $\III(K_2 \times G)$. As an example, we will determine the homotopy type of the closed neighborhood complex $\NNN(K_m \square K_n)$ of the cartesian product of the complete graph $K_m \square K_n$ by the homotopy type $\III(K_2 \times K_m \times K_n)$ determined by Antol\'in Camarena and Carnero Bravo \cite{BC} (Example~\ref{example cartesian}).

Next we mention the relationship between the closed neighborhood complex and the independence complex of the neighborhood hypergraph. In the previous work \cite{MW}, the author and Wakatsuki showed that the neighborhood complex $\NNN(\overline{G})$ of the complement $\overline{G}$ of $G$ coincides with the combinatorial Alexander dual (see Section~\ref{section preliminaries}) $\DDD(G)^\vee$ of the dominance complex $\DDD(G)$. In the case of the closed neighborhood complex, the independence complex $\III(\NNN_G)$ of the neighborhood hypergraph $\NNN_G$ appears instead of the dominance complex $\DDD(G)$.

The \emph{neighborhood hypergraph $\NNN_G$ of a finite graph $G$} is the hypergraph whose vertex set is $V(G)$, and whose set of hyperedges is the set of open neighborhoods of the vertices in $G$ (Definition~\ref{definition neighborhood hypergraph}).
It is straightforward to generalize the independence complex to hypergraphs (see Definition~\ref{definition hypergraph 2}), and we consider the independence complex $\III(\NNN_G)$ of the neighborhood hypergraph. In this case, a finite subset $\sigma$ of $V(G)$ is a simplex of $\III(\NNN_G)$ if and only if there is no $v \in V(G)$ such that $N_G(v) \subset \sigma$.
The following theorem holds:

\begin{theorem} \label{theorem neighborhood hypergraph}
Let $G$ be a finite graph. Then, $\III(\NNN_G)^\vee$ coincides with $\NNN[\overline{G}]$.
\end{theorem}

\subsection{Fundamental group of the closed neighborhood complex}

For a positive integer $k$, the author \cite{MatsushitaJMSUT} introduced the $k$-fundamental group $\pi_1^k(G,v)$ of a based graph $(G,v)$, and showed that the fundamental group of the neighborhood complex $\NNN(G)$ of $G$ is isomorphic to the subgroup $\pi_1^2(G,v)_{ev}$, which is of finite index in $\pi_1^2(G,v)$ and is called the even part. The study of $2$-fundamental groups of graphs was initiated by the author \cite{MatsushitaJMSUT} and has applications in several branches of graph theory, including the study of multiplicative graphs \cite{WrochnaJCTB1}, the homomorphism reconfiguration problem \cite{Matsushita_square-free} and the graph coloring problem in topological graph theory \cite{EM}.

The next result (Theorem~\ref{theorem B}) shows that the fundamental group of the closed neighborhood complex is isomorphic to the fundamental group of a graph introduced in the study of path homology. Path homology is a homology theory for directed graphs that was introduced in \cite{GLMY} and has been extensively studied in recent years.

In the study of path homology, Grigor'yan et al. \cite{GLMY} introduced the fundamental group $\pi_1^{\rm GLMY}(G,v)$ of a based graph $(G,v)$. Subsequently, Di et al. \cite{DIMZ} introduced the $k$-fundamental group $\pi_1^k[G,v]$ for each positive integer $k$, where $\pi_1^2[G,v]$ is isomorphic to $\pi_1^{\rm GLMY}(G,v)$. To distinguish this notion from the $k$-fundamental group of \cite{MatsushitaJMSUT}, we call the group $\pi_1^k[G,v]$ defined in \cite{DIMZ} the closed $k$-fundamental group in this paper. The closed $k$-fundamental groups are related to the magnitude homology \cite{HW} through the magnitude-path spectral sequence \cite{Asao} (see also \cite{KT}).

To state our theorem precisely, we introduce the closed $k$-neighborhood complex $\NNN^k[G]$ of a graph $G$ as follows. For a vertex $v$ of $G$, the \emph{closed $k$-neighborhood $N_G^k[v]$ of $v$ in $G$} is recursively defined by
\[ N_G^1[v] = N_G[v],\quad N_G^{k+1}[v] = \bigcup_{w \in N_G^k[v]} N_G[w].\]
Then, the \emph{closed $k$-neighborhood complex $\NNN^k[G]$ of $G$} is the abstract simplicial complex whose ground set is $V(G)$ and whose simplices are sets contained in some $k$-neighborhood. Note that the closed $1$-neighborhood complex $\NNN^1[G]$ coincides with the closed neighborhood complex $\NNN[G]$. Closed $k$-neighborhood complexes of hypercube graphs are studied in \cite{ASS}. Now, we are ready to state our next result:

\begin{theorem} \label{theorem B}
Let $(G,v)$ be a based graph and $k$ a positive integer. Then, there is an isomorphism
\[ \pi_1^{2k}[G,v] \cong \pi_1(\NNN^k[G],v)\]
that is natural with respect to based graph maps (see Definition~\ref{definition graph map}).
\end{theorem}

Although Di et al. \cite{DIMZ} and Kishimoto--Tong \cite{KT} considered the closed $k$-fundamental group for based directed graphs, the assumption of Theorem~\ref{theorem B} that $G$ is a graph is essential. It is straightforward to extend the definition of the closed $k$-neighborhood complexes to directed graphs (Definition~\ref{definition closed k}). However, in Appendix, we show that the fundamental group of this generalization of the closed $k$-neighborhood complex is not isomorphic to the closed $2k$-fundamental group in general.

\subsection{Organization}
The rest of this paper is organized as follows: In Section~\ref{section preliminaries}, we review several notions and facts of graphs and simplicial complexes that we need in this paper. In Section~\ref{section independence}, we prove Theorem~\ref{theorem A}, and discuss its applications. In Section~\ref{section neighborhood hypergraph}, we prove Theorem~\ref{theorem neighborhood hypergraph}. In Section~\ref{section fundamental}, we recall the definition of closed $k$-fundamental group and prove Theorem~\ref{theorem B}. In Appendix we consider the case of digraphs.

\subsection*{Acknowledgement}
The author was supported in part by JSPS KAKENHI Grant Numbers JP23K12975. The author thanks Yasuhiko Asao for useful information on magnitude homology and path homology. The author also thanks Henry Adams, Samir Shukla and Anurag Singh for pointing out the connection between this paper and their work \cite{ASS}.

\section{Preliminaries} \label{section preliminaries}

In this section, we review several definitions and facts of graphs and simplicial complexes, following \cite{GR, HN, Kozlovbook}.

\subsection{Graphs}

A \emph{directed graph} or a \emph{digraph} is a pair $G = (V(G), E(G))$ consisting of a (not necessarily finite) set $V(G)$ and a subset $E(G)$ of $V(G)^2 \setminus \Delta_{V(G)}$. Here, $\Delta_{V(G)}$ is the diagonal set of $V(G)$, i.e., $\Delta_{V(G)} = \{ (v,v) \mid v \in V(G)\}$. We call $V(G)$ the \emph{vertex set of $G$}, and $E(G)$ the \emph{edge set of $G$}. A digraph $G$ is called a \emph{graph} if $(x,y) \in E(G)$ implies $(y,x) \in E(G)$. Namely, our graphs are undirected, simple but not necessarily finite.

The \emph{complement $\overline{G}$ of a graph $G$} is the graph defined by
\[ V(\overline{G}) = V(G),\quad E(\overline{G}) = \{ (v,w) \in V(G) \times V(G) \mid \text{$v \ne w$ and $(v,w) \not\in E(G)$}\}.\]

The \emph{cartesian product $G \square H$ of graphs $G$ and $H$} is the graph defined as follows: The vertex set of $G \square H$ is $V(G) \times V(H)$, and two distinct elements $(v_1, w_1), (v_2, w_2)$ of $V(G)^2$ are adjacent in $G \square H$ if and only if one of the following conditions holds:
\begin{itemize}
\item $v_1 = v_2$ and $(w_1, w_2) \in E(H)$.

\item $(v_1, v_2) \in E(G)$ and $w_1 = w_2$.
\end{itemize}

The \emph{categorical product $G \times H$ of graphs $G$ and $H$} is the graph defined by $V(G \times H) = V(G) \times V(H)$ and 
\[ E(G \times H) = \{ ((v_1, w_1),(v_2, w_2)) \mid (v_1, v_2) \in E(G), (w_1, w_2) \in E(H)\}.\]

Let $K_n$ be the complete graph with $n$ vertices, i.e., $V(K_n) = \{ 1, \cdots, n\}$ and $E(K_n) = \{ (x,y) \mid x,y \in V(K_n),\ x \ne y\}$. Then, the categorical product $K_2 \times G$ is called the \emph{canonical double covering} or \emph{bipartite double covering}.

If $G$ is a non-bipartite connected graph, then $K_2 \times G$ is the unique connected double covering that is bipartite. If $G$ is bipartite, then $K_2 \times G$ is isomorphic to the disjoint union $G \sqcup G$ (see \cite{MatsushitaDM, Waller} for example).

\subsection{Simplicial complexes}

Let $X$ be a set. Let $2^X$ denote the power set of $X$, and $2^{(X)}$ the family of finite subsets of $X$. An \emph{(abstract) simplicial complex $K$ on the ground set $X$} is a subset $K \subset 2^{(X)}$ that satisfies the following condition: $\sigma \in K$ and $\tau \subset \sigma$ imply $\tau \in K$. In this paper, we assume that every simplicial complex contains the empty set $\emptyset$. A \emph{vertex of $K$} is $v \in X$ such that $\{ v\} \in K$. The set of vertices of $K$ is called the \emph{vertex set of $K$}, and is denoted by $V(K)$.

Now we recall the definition of the geometric realization $|K|$ of $K$. For a set $X$, we write $\RR^{(X)}$ to mean the $\RR$-module freely generated by $X$ with basis
\[ \{ e_x \mid x \in X\}.\]
We consider the topology of $\RR^{(X)}$ to be the inductive limit topology with respect to finite-dimensional subspaces. Namely, a subset $F$ of $\RR^{(X)}$ is closed if and only if for every finite subset $A \subset X$, $F \cap \RR^{(A)}$ is closed in $\RR^{(A)}$.

For a subset $S$ of a real vector space $V$, let $\conv(S)$ denote the convex hull of $S$. Let $K$ be a simplicial complex. For $\sigma \in K$, set
\[ |\sigma| = \conv \{ e_x \mid x \in \sigma \} \subset \RR^{(V(K))}.\]
Then, the \emph{geometric realization of $K$} is defined by
\[ |K| = \bigcup_{\sigma \in K} |\sigma|.\]
We regard $|K|$ as a topological space equipped with the subspace topology in $\RR^{(V(K))}$.

A map $f \colon V(K) \to V(L)$ is called a \emph{simplicial map} if $\sigma \in K$ implies $f(\sigma) \in L$. Then, the simplicial map $f$ induces a continuous map $|f| \colon |K| \to |L|$.

Let $K$ be a simplicial complex and $v$ a vertex of $K$. Define the \emph{deletion $K - v$} and the \emph{link $\link_K(v)$} to be the subcomplexes
\[ K - v = \{ \sigma \in K \mid v \not\in \sigma \},\quad \link_K(v) = \{ \sigma \in K - v \mid \sigma \cup \{ v\} \in K\}\]
of $K$. Then, $|K|$ is homeomorphic to the mapping cylinder of the inclusion $|\link_K(v)| \hookrightarrow |K - v|$. In particular, if $|K - v|$ is contractible, then $|K| \simeq \Sigma |\link_K(v)|$, where $\Sigma$ denotes the suspension.

Let $K$ be a simplicial complex on a finite ground set $X$. The \emph{combinatorial Alexander dual $K^\vee$ of $K$} is the simplicial complex whose ground set is $X$ and whose set of simplices is
\[ K^\vee = \{ \sigma \subset X \mid X - \sigma \not\in K\}.\]

\begin{theorem}[combinatorial Alexander duality theorem \cite{BT}] \label{theorem Alexander}
Let $K$ be a simplicial complex whose ground set $X$ is finite, and $R$ a commutative ring with unit. Then, there is the following isomorphism:
\[ \tilde{H}^{\# X - k - 3}(K ; R) \cong \tilde{H}_k(K; R).\]
\end{theorem}

\subsection{Graph complexes}

In this subsection, we review several simplicial complexes obtained from graphs. Throughout this subsection, $G$ denotes a graph.

A subset $\sigma$ of $V(G)$ is said to be \emph{independent in $G$} if no two elements in $\sigma$ are adjacent. The \emph{independence complex $\III(G)$ of $G$} is the simplicial complex whose ground set is $V(G)$ and whose set of simplices is the set of finite independent sets in $G$. For graphs $G$ and $H$, the independence complex $\III(G \sqcup H)$ of the disjoint union of $G$ and $H$ coincides with the join $\III(G) * \III(H)$ of $\III(G)$ and $\III(H)$.

A subset $\sigma$ of $V(G)$ is a \emph{clique in $G$} if any two elements in $\sigma$ are adjacent in $G$. Define the \emph{clique complex $\XXX(G)$ of $G$} to be the simplicial complex whose set of simplices is the set of finite cliques in $G$. It is clear that the clique complex $\XXX(\overline{G})$ of the complement of $G$ coincides with the independence complex $\III(G)$ of $G$, and that the independence complex $\III(\overline{G})$ of the complement of $G$ coincides with the clique complex $\XXX(G)$ of $G$.

\section{Independence complex of the canonical double covering} \label{section independence}

\subsection{Proof of Theorem~\ref{theorem A}}

The goal of this subsection is to prove Theorem~\ref{theorem A}, which states that $\Sigma \NNN[\overline{G}]$ is homotopy equivalent to $\III(K_2 \times G)$. Theorem~\ref{theorem A} is immediately deduced from the following theorem:

\begin{theorem} \label{theorem Nagel--Reiner}
Let $X$ and $Y$ be (possibly infinite) sets, and $\varphi \colon Y \to 2^X$ a map. Let $K$ be a simplicial complex on $X$ generated by $\{ \varphi(y) \mid y \in Y\}$, i.e., a finite subset $\sigma \subset X$ is a simplex of $K$ if and only if there is $y \in Y$ such that $\sigma \subset \varphi(y)$. Let $H$ be the bipartite graph with bipartition $X \sqcup Y$, where $x \in X$ and $y \in Y$ are adjacent if and only if $x \not\in \varphi(y)$. Then, we have $\III(H) \simeq \Sigma K$.
\end{theorem}

Nagel and Reiner showed that for every finite simplicial complex $K$, there is a finite bipartite graph $G$ such that $\Sigma K$ and $\III(G)$ are homotopy equivalent \cite[Proposition~6.2]{NR} (see also \cite{Jonsson} and \cite{BarmakAdvances}). The following proof of Theorem~\ref{theorem Nagel--Reiner} is a slightly modified proof of it. Specifically, Nagel and Reiner proved their theorem under the assumption that $X$ and $Y$ are finite sets and that $\varphi(y_1) \not\subset \varphi(y_2)$ holds for any distinct $y_1, y_2 \in Y$. However, as we will see below, these assumptions are not substantial.

\begin{proof}[Proof of Theorem~\ref{theorem Nagel--Reiner}]
We first consider the case that one of $X$ and $Y$ is empty. In this case we have $K = \emptyset$ and $\Sigma K = S^0$. On the other hand, $H$ is the complete bipartite graph with bipartition $X$ and $Y$. Hence we have $\Sigma K = S^0 \simeq \III(H)$.

Next we assume that neither $X$ nor $Y$ is empty. Define the continuous function $f \colon |\III(H)| \to \RR$ by
\[ f \bigg( \sum_{x \in X} a_x e_x + \sum_{y \in Y} b_y e_y\bigg) = \sum_{y \in Y} b_y.\]
Set $A_1 = f^{-1}([0, 2^{-1}])$ and $A_2 = f^{-1}([2^{-1}, 1])$. By the straight-line homotopies, $A_1$ and $A_2$ are homotopy equivalent to $f^{-1}(0)$ and $f^{-1}(1)$, respectively. Since $X$ and $Y$ are independent in $H$, $f^{-1}(0)$ and $f^{-1}(1)$ are contractible. By the straight-line homotopy, $A_1 \cap A_2$ is a deformation retract of $f^{-1}((0,2^{-1}])$. This means that $(A_1, A_1 \cap A_2)$ is an NDR-pair (see \cite[Definition~7.4]{Kozlovbook}), and hence the inclusion $A_1 \cap A_2 \hookrightarrow A_1$ is a cofibration (see \cite[Proposition~7.7]{Kozlovbook}). Similarly, the inclusion $A_1 \cap A_2 \hookrightarrow A_2$ is a cofibration. Hence, we have $\III(H) \simeq \Sigma (A_1 \cap A_2)$ (see \cite[Lemma~6.3]{NR} for example).

Thus it suffices to show that $A_1 \cap A_2$ is homotopy equivalent to $K$. To see this, we construct good covers of $K$ and $A_1 \cap A_2$, and apply the nerve lemma. 

Consider the covering $\{ U_y \mid y \in Y\}$ of $|K|$ where $U_y = |\Delta^{\varphi(y)}|$. Here $\Delta^{\varphi(y)}$ is the simplicial complex whose ground set is $\varphi(y)$ ans the set of simplices are $2^{(\varphi(y))}$. Clearly, this covering is a good cover of $|K|$.

Next, consider the covering $\{ V_y \mid y \in Y\}$ of $A_1 \cap A_2$ given by
\[ V_{y} = \bigg\{ \sum_{x \in \varphi(y)} a_x e_x + \sum_{z \in Y} b_z e_z \in |\III(H)| \mid \sum_{x \in \varphi(y)} a_x = \sum_{z \in Y} b_z = \frac{1}{2} \bigg\}.\]
Here, we regard $A_1 \cap A_2$ as a polyhedral complex (see \cite[Definition~2.39]{Kozlovbook}) whose faces are the convex polytopes described as $|\sigma| \cap f^{-1}(2^{-1})$ for $\sigma \in \III(H)$. Then, $V_y$ is a subcomplex of $A_1 \cap A_2$. Indeed, $V_y$ is the union of faces $|\sigma \cup \tau| \cap f^{-1}(2^{-1})$, where $\sigma \subset \varphi(y)$ is a finite subset of $X$ and $\tau \subset Y$ is a finite subset of $Y$ such that $y' \in \tau$ implies $\sigma \subset \varphi(y')$.

Next, we show that $V_{y_1} \cap \cdots \cap V_{y_l}$ is either empty or contractible. Indeed, if $V_{y_1} \cap \cdots \cap V_{y_l} \ne \emptyset$, then there is $x \in \varphi(y_1) \cap \cdots \cap \varphi(y_l)$. Define the homotopy $h_t \colon V_{y_1} \cap \cdots \cap V_{y_l} \to V_{y_1} \cap \cdots \cap V_{y_l}$ for $t \in [0,1]$ by
\[ h_t(w) = (1-t)w + \frac{t}{2} e_x + \frac{t}{2} e_{y_1}.\]
Then, $h_1$ is the constant map at $2^{-1} x + 2^{-1} y_1$. Hence, $V_{y_1} \cap \cdots \cap V_{y_l}$ is contractible. Hence, $\{ V_y \mid y \in Y\}$ is a good cover of $A_1 \cap A_2$. Since 
\[ U_{y_1} \cap \cdots \cap U_{y_k} = \emptyset \ \Leftrightarrow \ \varphi(y_1) \cap \cdots \cap \varphi(y_k) = \emptyset \ \Leftrightarrow \ V_{y_1} \cap \cdots \cap V_{y_k} = \emptyset,\]
the nerve complexes of $\{ U_y\}$ and $\{ V_y\}$ are isomorphic. Hence, the nerve lemma \cite[Remark~15.22]{Kozlovbook} completes the proof.
\end{proof}

\begin{remark} \label{remark unclear}
In the proof of \cite[Proposition~6.2]{NR}, the description of the good cover of $A_1 \cap A_2$ is somewhat ambiguous. It seems that they considered the good cover of $A_1 \cap A_2$ as follows: For a finite subset $\{ y_1, \cdots, y_l\} = \tau \subset Y$ of $Y$ with $\varphi(y_1) \cap \cdots \cap \varphi(y_l) = \bigcap_{y \in \tau} \varphi(y) \ne \emptyset$, set
\[ V'_{\tau} = \bigg\{ \sum_{x \in \varphi(y_1) \cap \cdots \cap \varphi(y_l)} a_x e_x + \sum_{y \in \tau}^l b_y e_y \mid \sum_{x \in \varphi(y_1) \cap \cdots \cap \varphi(y_l)} a_x = \sum_{y \in \tau} b_y = \frac{1}{2}\bigg\}.\]
However, in this case, $\tau \subsetneq \tau'$ does not imply that $V'_{\tau'} \subset V'_\tau$, and the nerve lemma cannot be applied.
\end{remark}

Theorem~\ref{theorem A} is immediately deduced from Theorem~\ref{theorem Nagel--Reiner}.

\begin{proof}[Proof of Theorem~\ref{theorem A}]
We apply Theorem~\ref{theorem Nagel--Reiner} to the following case. Let $G = (V(G), E(G))$ be a graph. Set $X = Y = V(G)$, and define $\varphi \colon Y \to 2^X$ by $\varphi(y) = N_{\overline{G}}[y]$. Then, we have $K = \NNN[\overline{G}]$ and $H = K_2 \times G$. Hence, Theorem~\ref{theorem Nagel--Reiner} implies $\Sigma \NNN[\overline{G}] \simeq \III(K_2 \times G)$.

If $G$ is bipartite, then $K_2 \times G \cong G \sqcup G$, and thus we have $\Sigma \NNN[\overline{G}] \simeq \III(K_2 \times G) = \III(G) * \III(G)$.
This completes the proof.
\end{proof}

\begin{proof}[Proof of Corollary~\ref{corollary A}]
This immediately follows from Theorem~\ref{theorem A} and the fact that $\XXX(\overline{G})$ coincides with $\III(G)$.
\end{proof}

By Theorem~\ref{theorem A}, the homotopy type of the independence complex $\III(K_2 \times G)$ of the canonical double covering of $G$ is determined by the closed neighborhood complex $\NNN[\overline{G}]$ of the complement of $G$.

\begin{example}
We consider the homotopy type of $\III(K_2 \times K_n)$. Since $\overline{K_n}$ is the graph consisting of $n$ isolated vertices, it follows from Theorem~\ref{theorem A} that
\[ \III(K_2 \times K_n) \simeq \Sigma \NNN[\overline{K_n}] = \bigvee_{n-1} S^1. \]
This homotopy equivalence is a part of \cite[Proposition~3.2]{GSS}.
\end{example}

\begin{example} \label{example Borsuk graph}
Here, we determine the homotopy type of the independence complex of the canonical double covering of the Borsuk graphs $\Bor(S^1 ; a)$ and $\Bor[S^1; a]$ of a circle.

Here we regard $S^1$ as $\RR / \ZZ$. Let $d$ be the distance on $S^1$ induced by $\RR$, i.e.,
\[ d(\alpha, \beta) = \inf \{ |x - y| \mid x \in \alpha, y \in \beta \}.\]
Let $a$ be a positive number. The \emph{Borsuk graphs} $\Bor(S^1 ; a)$ and $\Bor[S^1; a]$ are the graphs defined by $V(\Bor(S^1; a)) = V(\Bor[S^1; a]) = S^1$ and
\[ E(\Bor(S^1 ; a)) = \{ (\alpha, \beta) \in S^1 \times S^1 \mid d(\alpha, \beta) > a\},\]
\[ E(\Bor[S^1 ; a]) = \{ (\alpha, \beta) \in S^1 \times S^1 \mid d(\alpha, \beta) \ge a\}.\]
The Borsuk graphs of a circle is used to define the circular chromatic number. The circular chromatic number $\chi_c(G)$ is a refinement of the chromatic number $\chi(G)$, and we refer the reader to \cite{Zhu1, Zhu2}. Regarding the Borsuk graphs of other metric spaces, see \cite{AJM} and references therein.

Note that the independence complexes of $\Bor (S^1 ; a)$ and $\Bor[S^1 ; a]$ coincide with certain Vietoris--Rips complexes of a circle, and hence their homotopy types were determined in \cite[Main result]{AA}.

Now, we determine the homotopy types of the independence complexes of $K_2 \times \Bor(S^1; a)$ and $K_2 \times \Bor[S^1 ; a]$. The closed neighborhood complex of the complement of $\Bor(S^1 ; a)$ (or $\Bor [S^1 ; a]$) is the \v{C}ech complex $\CCC[S^1 ; \frac{1}{2} - a]$ (or $\CCC(S^1 ; \frac{1}{2} - a)$, respectively). For the definitions of $\CCC(X,d)$ and $\CCC[X,d]$ of a metric space $(X,d)$, see Section~\ref{section introduction}. The homotopy types of $\CCC(S^1; r)$ and $\CCC[S^1 ; r]$ are determined by \cite[Main result]{AA}:
\[ \CCC(S^1; r) \simeq S^{2l + 1} \quad \left( \frac{l}{2(l+1)} < r \le \frac{l+1}{2(l+2)}\right),\]
\[ \CCC[S^1 ; r] \simeq \begin{cases}
S^{2l + 1} & \left( \frac{l}{2(l+1)} < r < \frac{l+1}{2(l+2)} \right) \\
\bigvee_\mathfrak{c} S^{2l} & r = \frac{l}{2(l+1)}
\end{cases}\]
Here, $\mathfrak{c}$ means the cardinality of the continuum. Thus, Theorem~\ref{theorem A} implies
\[ \III(K_2 \times \Bor(S^1 ; a)) \simeq \Sigma \CCC \Big[S^1 ; \frac{1}{2} - a \Big] \simeq \begin{cases}
S^{2l + 2} & \frac{l}{2(l+1)} < \frac{1}{2} - a < \frac{l + 1}{2(l+2)} \\
\bigvee_{\mathfrak{c}} S^{2l + 1} & \frac{1}{2} - a = \frac{l}{2(l+1)},
\end{cases}\]
\[ \III(K_2 \times \Bor[S^1 ; a]) \simeq \Sigma \CCC \Big(S^1; \frac{1}{2} - a \Big) \simeq S^{2l + 2} \quad \left( \frac{l}{2(l+1)} < \frac{1}{2} - a \le \frac{l+1}{2(l+2)}\right).\]
\end{example}

In the above examples, we determine the homotopy type of $\III(K_2 \times G)$ from the homotopy type of $\NNN[\overline{G}]$. Conversely, in the following example we determine the homotopy type of the closed neighborhood complex $\NNN[G]$ from the homotopy type of $\III(K_2 \times \overline{G})$.

\begin{example} \label{example cartesian}
Goyal--Shukla--Singh \cite{GSS} conjectured that the independence complex of $K_2 \times K_3 \times K_n$ is homotopy equivalent to a wedge of $(n-1)(3n-2)$ spheres of dimension $3$. Antolín Camarena and Carnero Bravo \cite{BC} resolved this conjecture by showing $\III(K_2 \times K_m \times K_n) \simeq \bigvee_{f(m,n)} S^3$, where
\[ f(m,n) = \frac{(n-1)(m-1)(mn-2)}{2}.\]
Note that the complement $\overline{K_m \times K_n}$ of $K_m \times K_n$ coincides with the cartesian product $K_m \square K_n$. By Theorem~\ref{theorem A}, we have
\[ \III(K_2 \times K_m \times K_n) \simeq \Sigma \NNN[K_m \square K_n].\]
Therefore we have
\[ \tilde{H}_k(\NNN[K_m \square K_n] ; \ZZ) \cong \begin{cases}
\ZZ^{f(m,n)} & (k = 2) \\
0 & \text{(otherwise).}
\end{cases}\]
Indeed, it is not difficult to see that $\NNN[K_m \square K_n]$ is simply-connected (see Example~\ref{example edge-path}). Hence, it follows from \cite[Example~4.34]{Hatcher} that
\[ \NNN[K_m \square K_n] \simeq \bigvee_{f(m,n)} S^2.\]
\end{example}

\section{Independence complex of neighborhood hypergraph} \label{section neighborhood hypergraph}

In this section, we show Theorem~\ref{theorem neighborhood hypergraph}, which states that for a finite graph $G$ the closed neighborhood complex $\NNN[\overline{G}]$ of the complement of $G$ coincides with the combinatorial Alexander dual $\III(\NNN_G)^\vee$ of the neighborhood hypergraph $\NNN_G$ of $G$. We also discuss the independence complex of the neighborhood hypergraph of a forest (Example~\ref{example forest}). Throughout this section, every graph is assumed to be finite.

\begin{definition} \label{definition hypergraph}
A \emph{hypergraph $H$} is a pair $(V(H), E(H))$ consisting of a finite set $V(H)$ and a family $E(H) \subset 2^{V(H)}$ of subsets of $V(H)$. A subset $\sigma$ of $V(H)$ is said to be \emph{independent} if there is no $e \in E(H)$ such that $e \subset \sigma$.
\end{definition}

\begin{definition}[see \cite{MW, Taylan} for example] \label{definition hypergraph 2}
The \emph{independence complex $\III(H)$ of the hypergraph $H$} is the simplicial complex whose ground set is $V(H)$ and whose set of simplices is the set of independent sets in $H$.
\end{definition}

In the previous work \cite{MW}, the author and Wakatsuki showed that there is a close relationship between the neighborhood complex $\NNN(\overline{G})$ of the complement of $G$ and the dominance complex $\DDD(G)$. Recall that the dominance complex is the simplicial complex defined as follows: The ground set of $\DDD(G)$ is $V(G)$ and $\sigma \subset V(G)$ is a simplex of $\DDD(G)$ if and only if $V(G) - \sigma$ is a dominating set. Here, $\tau \subset V(G)$ is said to be \emph{dominating in $G$} if for every $x \in V(G) - \tau$ there is $y \in \tau$ that is adjacent to $x$.

Indeed, it is easy to see that $\DDD(G)$ is the independence complex of the following hypergraph $\DDD_G$: the vertex set of $\DDD_G$ is $V(G)$ and the set $E(\DDD_G)$ of hyperedges is $\{ N_G[v] \mid v \in V(G)\}$. Then, the author and Wakatsuki showed the following theorem:

\begin{theorem}[{\cite[Theorem~1.1]{MW}}] \label{theorem previous}
Let $G$ be a finite graph. Then, $\DDD(G)^\vee$ coincides with $\NNN(\overline{G})$.
\end{theorem}

We consider an analogue of the above theorem for the closed neighborhood complex. This leads to a close relationship between the closed neighborhood complex of the complement of a graph and the independence complex of the neighborhood hypergraph:

\begin{definition}[see \cite{BGZ} for example] \label{definition neighborhood hypergraph}
Let $G$ be a finite graph. The \emph{neighborhood hypergraph $\NNN_G$} is defined by $V(\NNN_G) = V(G)$ and $E(\NNN_G) = \{ N_G(v) \mid v \in V(G)\}$.
\end{definition}

The neighborhood hypergraph has been studied in the following realization problem: Given a hypergraph $H$, determine whether there is a graph $G$ such that the neighborhood hypergraph $\mathcal{N}_G$ is isomorphic to $H$.  This problem, originally suggested by Sós \cite{Sos}, was observed by Babai \cite{Babai} to be at least as hard as the graph Isomorphism Problem. Later, Lalonde \cite{Lalonde1, Lalonde2} proved that this realization problem is NP-hard. For further references, we refer the reader to \cite{BGZ}.

Thus, Theorem~\ref{theorem neighborhood hypergraph} is an analogue of Theorem~\ref{theorem previous} in the case of the closed neighborhood complex.

\begin{proof}[Proof of Theorem~\ref{theorem neighborhood hypergraph}]
Let $\sigma$ be a subset of $V(G)$. Consider the following conditions on $\sigma$:
\begin{enumerate}[(1)]
\item $\sigma$ is a simplex of $\III(\NNN_G)^\vee$.

\item $V(G) \setminus \sigma$ is a non-face of $\III(\NNN_G)$.

\item There is $v \in V(G)$ such that $N_G(v)$ is contained in $V(G) \setminus \sigma$.

\item There is $v \in V(G)$ such that $\sigma$ is contained in $N_{\overline{G}}[v] = V(G) \setminus N_G(v)$.

\item $\sigma$ is a simplex of $\NNN[\overline{G}]$.
\end{enumerate}
For $i = 1, \cdots, 4$, the equivalence $(i) \Leftrightarrow (i+1)$ clearly holds. This completes the proof.
\end{proof}

\begin{example} \label{example forest}
Here we show that for a forest $F$, $\III(\NNN_F)$ is contractible or homotopy equivalent to a sphere.

Let $F$ be a forest. Suppose that $F$ has a connected component $F_0$ having at least $3$ vertices. Let $v \in F_0$ be a leaf. Then, there is $w \in V(F_0)$ such that $N_F(v) \subset N_F(w)$. Then, the deletion of $w$ in $\III(\NNN_F)$ is a cone with apex $v$, and the link of $w$ coincides with $\III(\NNN_{F - w})$. Hence, we have $\III(\NNN_F) \simeq \Sigma \III(\NNN_{F - w})$.

Repeating this procedure, we have that $\III(\NNN_F)$ is homotopy equivalent to a repeated suspension of $\III(\NNN_{F'})$, where every connected component of $F'$ is isomorphic to $K_2$ or $K_1$. If $F'$ contains an isolated vertex $v$, then $\III(\NNN_{F'})$ is a cone with apex $v$, and in this case $\III(\NNN_{F'})$ is contractible.

On the other hand, suppose that every connected component of $F'$ is isomorphic to $K_2$. In this case, $\III(\NNN_{F'})$ is the empty space. Since a repeated suspension of the empty space is a sphere, we conclude that $\III(\NNN_F)$ is homotopy equivalent to a sphere $S^m$ for some $m$. We can show that this integer $m$ is $\nu(F) - 2 \gamma(F) - 1$, where $\nu(F)$ is the number of vertices of $F$ and $\gamma(F)$ is the domination number of $F$. Recall that the \emph{domination number $\gamma(G)$ of a graph $G$} is the minimum cardinality of dominating sets in $G$. Indeed, this calculation $m = \nu(F) - 2 \gamma(F) - 1$ can be directly proved by induction on the number of vertices, but here we relate this to the independence complex $\III(F)$ of $F$. By Theorems~\ref{theorem A} and \ref{theorem neighborhood hypergraph}, we have
\[ \Sigma \big( \III(\NNN_F)^\vee \big) = \Sigma \NNN[\overline{F}] \simeq \III(K_2 \times F) \cong \III(F) * \III(F).\]
It is known that the independence complex $\III(F)$ of a forest $F$ is contractible or homotopy equivalent to $S^{\gamma(F) - 1}$ (see \cite{MT}). Since in this case $\Sigma (\III(\NNN_F)^\vee)$ is not contractible, we have $\Sigma \big( \III(\NNN_F)^\vee \big) \simeq S^{2 \gamma(F) - 1}$. Thus the Alexander duality theorem (Theorem~\ref{theorem Alexander}) shows
\[ \tilde{H}_k(\III(\NNN_F); \ZZ) \cong \begin{cases}
\ZZ & (k = \nu(F) - 2\gamma(F) - 1) \\
0 & \text{(otherwise)}
\end{cases}\]
Since we have already seen that $\III(\NNN_F)$ is homotopy equivalent to a sphere, we conclude
\[\III(\NNN_F) \simeq S^{\nu(F) - 2\gamma(F) - 1}.\]
\end{example}

\section{Fundamental groups of closed neighborhood complexes} \label{section fundamental}

The goal of this section is to prove Theorem~\ref{theorem B}, which states that the fundamental group of the $k$-neighborhood complex $(\NNN^k[G], v)$ is naturally isomorphic to the closed $2k$-fundamental group $\pi_1^{2k}[G,v]$ introduced by Di--Ivanov--Mukoseev--Zhang \cite{DIMZ}. In Subsection~\ref{subsection graph maps}, we discuss the morphisms in our category of graphs. In Subsection~\ref{subsection KY}, we recall the combinatorial description of the closed $k$-fundamental group provided by Kishimoto--Tong \cite{KT}. In Subsection~\ref{subsection another formulation}, we give another formulation of the closed $k$-fundamental group for graphs. In Subsection~\ref{subsection edge-path}, we recall the edge-path group used in the proof of Theorem~\ref{theorem B}. In Subsection~\ref{subsection proof b}, we prove Theorem~\ref{theorem B}.

\subsection{Graph maps} \label{subsection graph maps}

In this subsection, we describe our category of graphs to explain the meaning of the naturality condition in Theorem~\ref{theorem B}. Recall that a function $f \colon V(G) \to V(H)$ is called a \emph{homomorphism} if $(x,y) \in E(G)$ implies $(f(x), f(y)) \in E(H)$. In many cases, the class of homomorphisms is employed as the class of morphisms in the category of graphs (see \cite{HN} for example).

Note that a function $f \colon V(G) \to V(H)$ is a homomorphism if and only if $f(N_G(v)) \subset N_H(f(v))$ for every $v \in V(G)$. Thus, from the perspective of the closed neighborhood, it is also natural to consider the following notion as a morphism. In fact, such a notion is sometimes adopted as the morphisms of the category of graphs (see \cite[Definition~2.2]{KT} for example).

\begin{definition} \label{definition graph map}
Let $G$ and $H$ be graphs. A \emph{map} or a \emph{graph map from $G$ to $H$} is a function $f \colon V(G) \to V(H)$ such that $v \in V(G)$ implies $f(N_G[v]) \subset N_H[f(v)]$. In other words, whenever $(x,y) \in E(G)$, $f(y)$ is adjacent to $f(x)$ or coincides with $f(x)$.

A \emph{based graph} is a pair $(G,v)$ consisting of a graph $G$ and a vertex $v$. A graph map $f \colon G \to H$ between based graphs $(G,v)$ and $(H,w)$ is said to be \emph{basepoint preserving}, or simply \emph{based} if $f(v) = w$.
\end{definition}

Note that every graph map is a homomorphism. Unless otherwise stated, we consider graph maps as the morphisms in the category of graphs in this paper.

The following lemma shows that the closed $k$-neighborhood complex $\NNN^k[G]$ is functorial with respect to graph maps.

\begin{lemma}
Let $f \colon G \to H$ be a graph map, and $k$ a positive integer. Then, the following hold:
\begin{enumerate}[(1)]
\item For $v \in V(G)$, $f(N^k_G[v]) \subset N^k_H[f(v)]$ holds.

\item The map $f \colon V(G) \to V(H)$ is a simplicial map from $\NNN^k[G]$ to $\NNN^k[H]$.
\end{enumerate}
\end{lemma}
\begin{proof}
By definition, we have $f(N^1_G[v]) \subset N^1_H[f(v)]$. Then, by induction, we have
\[ f(N_G^{i+1}[v]) = \bigcup_{w \in N^i_G[v]} f(N_G[w]) \subset \bigcup_{w \in N^i_G[v]} N_H[f(w)] \subset \bigcup_{y \in N^i_H[f(v)]} N_H[y] = N_H^{i+1}[f(v)].\]
This completes the proof of (1). It is clear that (2) follows from (1).
\end{proof}

\subsection{Combinatorial description by Kishimoto--Tong} \label{subsection KY}

In this subsection we recall the combinatorial description of the closed $k$-fundamental group $\pi_1^k[X,x_0]$ for a based digraph $(X, x_0)$ by Kishimoto--Tong \cite{KT}.

Let $X$ and $Y$ be digraphs. A map $f \colon V(X) \to V(Y)$ is called a \emph{digraph map} if $(x,y) \in E(X)$ implies $(f(x), f(y)) \in E(Y)$ or $f(x) = f(y)$. Let $\JJJ_0$ be the set consisting of the digraph with a single vertex $0$. For a positive integer $n$, let $\JJJ_n$ denote the set of digraphs with vertex set $\{0,1,\cdots, n\}$ having exactly one edge $(i,i+1)$ or $(i+1,i)$ for each $i = 0,1,\cdots, n-1$ and no other edges.

Let $\vec{I}_n$ be the digraph such that $\vec{I}_n \in \JJJ_n$ and
\[ E(\vec{I}_n) = \{ (i, i+1) \mid i = 0,1,\cdots, n-1\}.\]

A \emph{path of length $n$} in a digraph $X$ is a digraph map $I_n \to X$ for some $I_n \in \JJJ_n$ with $n \ge 0$. A directed path in a digraph $X$ is a path $\vec{I}_n \to X$.

For $k \ge 1$, define the digraph $\Gamma_k$ by
\[ V(\Gamma_k) = \{ u_0 = v_0, u_1, v_1, u_2, v_2, \cdots, u_{k-1}, v_{k-1}, u_k = v_k\},\]
\[ E(\Gamma_k) = \{ (u_i, u_{i+1}), (v_i, v_{i+1}) \mid i = 0,1, \cdots, k-1\}.\]
Figure~\ref{figure gamma} depicts $\Gamma_k$ for $k \ge 2$. Let $x$ be a vertex of $\Gamma_k$. For $k = 0$, let $\rho_x$ be the constant path $\rho_x \colon I \to \Gamma_0$ with $I \in \JJJ_1$. For $r = 1$, let $\rho_x \colon I_2 \to \Gamma_1$ with $I_2 \in \JJJ_2$ be the unique non-constant path such that $\rho_x(0) = x = \rho_x(2)$. For $k \ge 2$, let $\rho_x \colon I_{2k} \to \Gamma_k$ with $I_{2k} \in \JJJ_{2k}$ be the reduced path with $\rho_x(0) = x = \rho_x(2k)$ that is clockwise with respect to Figure~\ref{figure gamma}.

For $I_m \in \JJJ_m$ and $I_n \in \JJJ_n$, the concatenation $I_m + I_n \in \JJJ_{m+n}$ is defined in an obvious way. Then, for a pair of paths $f \colon I_m \to X$ and $g \colon I_n \to X$ with $f(m) = g(0)$, define the \emph{concatenation $f \cdot g \colon I_m + I_n \to X$ of $f$ and $g$} by
\[ f \cdot g(i) = \begin{cases}
f(i) & (0 \le i \le m) \\
g(i - m) & (m \le i \le m + n).
\end{cases}\]

A \emph{based digraph} is a pair $(X,x_0)$ consisting of a digraph $X$ and a vertex $x_0$ of $X$. For a based digraph $(X,x_0)$, let $\vec{\Omega}(X,x_0)$ be the set of paths of $X$ whose initial and terminal points are $x_0$. For $f,g \in \vec{\Omega}(X,x_0)$, we write $f \to_k g$ if one of the following conditions holds:
\begin{enumerate}[(1)]
\item $f = g$.

\item There are decompositions $f = f_1 \cdot f_2$ and $g = f_1 \cdot (h \circ \rho_x) \cdot f_2$, where $f_1$ and $f_2$ are paths in $X$ and $h \colon \Gamma_s \to X$ is a digraph map with $s = 0$ or $k$,  $x$ is a vertex of $\Gamma_k$ such that $h(x)$ is the terminal point of $f_1$.
\end{enumerate}

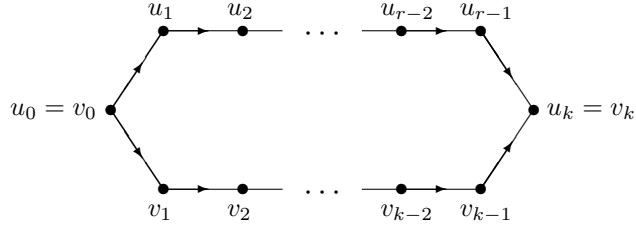
\begin{figure}
\centering
\begin{picture}(150,80)(0,-8)
\put(0,30){\circle*{4}}
\put(0,30){\line(2,3){20}}
\put(0,30){\line(2,-3){20}}

\put(0,30){\vector(2,3){12}}
\put(0,30){\vector(2,-3){12}}

\put(20,0){\vector(1,0){17}}
\put(20,60){\vector(1,0){17}}

\put(20,0){\circle*{4}}
\put(50,0){\circle*{4}}
\put(110,0){\circle*{4}}
\put(140,0){\circle*{4}}

\put(20,0){\line(1,0){45}}
\put(20,60){\line(1,0){45}}
\put(140,0){\line(-1,0){45}}
\put(140,60){\line(-1,0){45}}

\put(73,-4){$\cdots$}
\put(73,56){$\cdots$}

\put(20,60){\circle*{4}}
\put(50,60){\circle*{4}}
\put(110,60){\circle*{4}}
\put(140,60){\circle*{4}}

\put(160,30){\circle*{4}}
\put(160,30){\line(-2,3){20}}
\put(160,30){\line(-2,-3){20}}

\put(140,0){\vector(2,3){12}}
\put(140,60){\vector(2,-3){12}}

\put(110,0){\vector(1,0){17}}
\put(110,60){\vector(1,0){17}}

\put(-38,28){\footnotesize $u_0 = v_0$}

\put(14,66){\footnotesize $u_1$}
\put(44,66){\footnotesize $u_2$}

\put(102,66){\footnotesize $u_{r-2}$}
\put(132,66){\footnotesize $u_{r-1}$}

\put(14,-10){\footnotesize $v_1$}
\put(44,-10){\footnotesize $v_2$}

\put(102,-10){\footnotesize $v_{k-2}$}
\put(132,-10){\footnotesize $v_{k-1}$}

\put(165,28){\footnotesize $u_k = v_k$}

\end{picture}
\caption{Digraph $\Gamma_k$}
\label{figure gamma}
\end{figure}

Let $\approx_k$ be the smallest equivalence relation containing the relation $f \to_k g$. Then, it is straightforward to see that the quotient set $\vec{\pi}_1^k(X,x_0) = \vec{\Omega}(X, x_0) / \approx_k$ is a group, whose group operation is the map induced by the concatenation of paths. This is the formulation of the closed $k$-fundamental group of the based digraph $(X,x_0)$.

The following lemma shows that if $G$ is a graph, then the equivalence class of $\approx_k$ does not depend on the choice of the directions of paths:

\begin{lemma} \label{lemma direction}
Let $(G,v)$ be a based graph. Let $f$, $g \in \JJJ_n$ such that $f(i) = g(i)$ for every $i = 0,1, \cdots, n$. Then, $f \approx_k g$ for $k \ge 2$.
\end{lemma}
\begin{proof}
Since $f \approx_2 g$ implies $f \approx_k g$ (see \cite[Lemma~3.7]{KT}), it suffices to see that $f \approx_2 g$.

Let $I_f$ and $I_g$ be the domains of $f$ and $g$, respectively. Set $x_i = f(i) = g(i)$. We can assume that the directions of the edges of $I_f$ and $I_g$ only differ  between $i-1$ and $i$, and that $(i-1) \to i$ in $I_f$. Then, we have
\begin{eqnarray*}
f &=& x_0\cdots x_{i-1}\to x_i \cdots x_n\\
& \approx_2 & x_0 \cdots x_{i-1} \to x_i \to x_{i-1} \leftarrow x_i \cdots x_n \\
& \approx_2 & x_0 \cdots x_{i-1} \to x_i \to x_{i-1} \leftarrow x_{i-1} \leftarrow x_{i-1} \leftarrow x_i \cdots x_n \\
& \approx_2 & x_0 \cdots x_{i-1} \leftarrow x_i \cdots x_n = g.
\end{eqnarray*}
This completes the proof.
\end{proof}

\subsection{Another formulation} \label{subsection another formulation}

In this subsection, we provide another formulation of the closed $k$-fundamental group $\pi_1^k[G,v]$ of a based graph $(G,v)$ for $k \ge 2$, which is slightly different from the formulation by Kishimoto--Tong.

Let $G$ be a graph. Let $P_n$ be the path graph consisting of $n + 1$ vertices, i.e., $V(P_n) = \{ 0,1, \cdots, n\}$ and $E(P_n) = \{ (x,y) \in V(P_n)^2 \mid |x - y| = 1\}$. A \emph{path in $G$ of length $n$} is a graph map $\gamma \colon P_n \to G$. If $\gamma(i) = x_i$ for $i \in \{0,1, \cdots, n\}$, then we denote the path $\gamma$ by $(x_0, \cdots, x_n)$. Let $\Omega[G, v, w]$ denote the set of paths in $G$ joining $v$ to $w$. For a path $\gamma$, we denote the length of $\gamma$ by $l(\gamma)$. Note that for a pair of vertices $v$ and $w$ of $G$, $w \in N^k_G[v]$ if and only if there is a path joining $v$ to $w$ of length $k$.

For a pair of paths $\gamma \colon P_m \to G$ and $\delta \colon P_n \to G$ with $\gamma(m) = \delta(0)$, the \emph{concatenation of $\gamma$ and $\delta$} is the path $\gamma \cdot \delta \colon P_{m + n} \to G$ defined by
\[ \gamma \cdot \delta(i) = \begin{cases}
\gamma(i) & (i \le m)\\
\delta(i - m) & (i \ge m).
\end{cases}\]
For a path $\gamma \colon P_n \to G$, define the \emph{reverse $\bar{\gamma} \colon P_n \to G$ of $\gamma$} by $\bar{\gamma}(i) = \gamma(n-i)$.

Let $k$ be an integer greater than $1$. Let $\simeq_k$ be the smallest equivalence relation on $\Omega [G, v, w]$ containing the following two relations:

\begin{enumerate}
\item[(A)] $l(\gamma') = l(\gamma) + 1$ and there is $x \in \{ 0,1,\cdots, l(\gamma)\}$ such that $i \le x$ implies $\gamma(i) = \gamma'(i)$ and $i \ge x$ implies $\gamma(i) = \gamma'(i+1)$.

\item[(B)$_k$] $l(\gamma) = l(\gamma')$ and there is $i_0 \in \{ 0,1, \cdots, l(\gamma)\}$ such that $\gamma(i) = \gamma'(i)$ for $i \le i_0$ and for $i \ge i_0 + k$.
\end{enumerate}
Define $\pi_1^k[G,v,w]$ to be the quotient set $\Omega[G, v, w] / \simeq_k$. For a path $\gamma \in \pi_1^k[G,v,w]$, we write $[\gamma]_k$ to indicate the equivalence class of $\simeq_k$ containing $\gamma$, and call it the \emph{closed $k$-homotopy class of $\gamma$}.

\begin{remark}
Consider the following condition on $\Omega [G,v, w]$:
\begin{enumerate}
\item[(B)$'_k$] $l(\gamma) = l(\gamma')$ and $\# \{ i \mid \gamma(i) \ne \gamma'(i)\} < k$.
\end{enumerate}
It is straightforward to see that the smallest equivalence relation containing (A) and (B)$'_k$ coincides with $\simeq_k$.
\end{remark}

Let $(G,v)$ be a based graph. We write $\pi_1^k[G,v]$ to mean $\pi_1^k[G,v,v]$. Then, the concatenation of paths induces a map
\[ \pi_1^k[G,v] \times \pi_1^k[G,v] \to \pi_1^k[G,v], \quad [\gamma]_k \cdot [\delta]_k \mapsto [\gamma \cdot \delta]_k.\]
It is straightforward to see that this is a group operation. Note that for $[\alpha] \in \pi_1^k[G,v]$, the inverse is given by $[\bar{\alpha}]$. We call this group $\pi_1^k[G,v]$ the \emph{closed $k$-fundamental group of the based graph $(G,v)$}.

\begin{proposition} \label{proposition KT isomorphism}
Let $k$ be an integer greater than $1$. Then, for a based graph $(G,v)$, there is an isomorphism
\[ \pi_1^k[G,v] \to \vec{\pi}_1^k[G,v],\]
which is natural with respect to graph maps.
\end{proposition}
\begin{proof}
Let $\Phi \colon \pi_1^k[G,v] \to \vec{\pi}_1^k[G,v]$ be the map induced by the correspondence of a path $\gamma \colon P_n \to G$ to $\hat{\gamma} \colon \vec{I}_n \to G$, where $\hat{\gamma}(i) = \gamma(i)$ for every $i = 0,1, \cdots, n$. Let $\Psi \colon \vec{\pi}_1^k[G,v] \to \pi_1^k[G,v]$ be the map induced by the correspondence of $f \colon I_n \to G$ to $\bar{f} \colon P_n \to G$, where $\bar{f}(i) = f(i)$ for every $i = 0,1, \cdots, n$. It is straightforward to see that $\Phi$ and $\Psi$ are well-defined group homomorphisms, and Lemma~\ref{lemma direction} implies that $\Psi$ is the inverse of $\Phi$.
\end{proof}

\subsection{Edge-path group} \label{subsection edge-path}

In this subsection, we recall the edge-path group of simplicial complexes. For details, we refer the reader to \cite{Spanier}.

Let $K$ be a simplicial complex. An \emph{edge-path of length $n$ in $K$} is a sequence of vertices $(v_0, \dots, v_n)$ of $K$ such that for each $i = 1, \dots, n$, the set $\{v_{i-1}, v_i\}$ is a simplex in $K$. For an edge-path $\gamma = (v_0, \dots, v_n)$, $v_0$ is called the \emph{initial vertex of $\gamma$} and $v_n$ is called the \emph{terminal vertex of $\gamma$}. Let $v$ be a vertex of $K$. Let $E(K,v)$ denote the set of all edge-paths that start at $v$ and end at $v$. We define the smallest equivalence relation ``$\simeq$'' on $E(K,v)$ containing the following two conditions:

\begin{enumerate}[(a)]
\item If $v_{i-1} = v_i$, then $(v_0, \dots, v_n) \simeq (v_0, \dots, v_{i-1}, v_{i+1}, \dots, v_n)$.

\item If $\{ v_{i-1}, v_i, v_{i+1}\} \in K$, then $(v_0, \dots, v_n) \simeq (v_0, \dots, v_{i-1}, v_{i+1}, \dots, v_n)$.
\end{enumerate}

Two edge-paths $\gamma$ and $\delta$ are said to be \emph{homotopic} if $\gamma \simeq \delta$. Let $\EEE(K,v)$ denote the quotient set $E(K, v) / \simeq$. The set $\EEE(K,v)$ forms a group under the concatenation of paths, and this group is called the \emph{edge-path group of a based simplicial complex $(K,v)$}.

\begin{theorem}[{see Spanier \cite[Corollary~3.6.17]{Spanier}}] \label{theorem edge-path}
For a based simplicial complex $(K,v)$, $\pi_1(|K|, v)$ and $\EEE(K,v)$ are naturally isomorphic.
\end{theorem}

\begin{example} \label{example edge-path}
Let $m$ and $n$ be positive integers. As was mentioned in Example~\ref{example cartesian}, we now show that $\NNN[K_m \square K_n]$ is simply-connected. Consider the basepoint of $K_m \square K_n$ as $(1,1)$, and let $(v_0, \cdots, v_l) \in E(\NNN[K_m \square K_n], (1,1))$. Suppose that $l \ge 2$. It suffices to show that for $j \in \{ 1, \cdots, l-1\}$, $(v_0, \cdots, v_l) \simeq (v_0, \cdots, \hat{v}_j, \cdots, v_l)$.

Set $v_i = (x_i, y_i)$ for $i = 0,1, \cdots, n$ and $v = (x_{j-1}, y_j)$. Then, we have
\begin{eqnarray*}
(v_0, \cdots, v_l) & \simeq & (v_0, \cdots, v_{j-1}, v, v_j, v_{j+1}, \cdots, v_l) \\
& \simeq & (v_0, \cdots, v_{j-1}, v, v_{j+1}, \cdots, v_l) \\
& \simeq & (v_0, \cdots, v_{j-1}, v_{j+1}, \cdots, v_l).
\end{eqnarray*}
Here we use
\[ v_{j-1}, v, v_j \in N_G[(x_{j-1}, y_j)],\quad v, v_j,  v_{j+1} \in N_G[(x_{j+1}, y_j)], \quad v_{j-1}, v, v_{j+1} \in N_G[(x_{j-1}, y_{j+1})],\]
where $G = K_m \square K_n$. This completes the proof.
\end{example}

\subsection{Proof of Theorem~\ref{theorem B}} \label{subsection proof b}

We now begin the proof of Theorem~\ref{theorem B}, which states that $\pi_1(\NNN^k[G], v)$ is naturally isomorphic to $\pi_1^{2k}[G,v]$.

\begin{proof}[Proof of Theorem~\ref{theorem B}]
By Theorem \ref{theorem edge-path}, $\pi_1(\NNN^k[G], v)$ is isomorphic to $\EEE(\NNN^k[G],v)$. Hence, it suffices to show that $\EEE(\NNN^k[G], v)$ and $\pi_1^{2k}[G,v]$ are isomorphic.

We define a group homomorphism $\Phi \colon \EEE(\NNN^k[G], v) \to \pi_1^{2k}[G,v]$ as follows. First, we construct a map $\tPhi \colon E(\NNN^k[G], v) \to \pi_1^{2k}[G,v]$. Let $\gamma = (v_0, \dots, v_n) \in E(\NNN^k[G], v)$. Then, for each $i = 1, \dots, n$ there exists a vertex $w_i$ such that $v_{i-1}, v_i \in N^k_G[w_i]$. Thus, we can choose a path $\alpha_i$ of length $k$ joining $v_{i-1}$ to $w_i$ and a path $\beta_i$ of length $k$ joining $w_i$ to $v_i$. By the definition of the $2k$-homotopy $\simeq_{2k}$ (Subsection~\ref{subsection another formulation}), the $2k$-homotopy class $[\alpha_1 \cdot \beta_1 \cdots \alpha_n \cdot \beta_n]_{2k}$ is independent of the choices of $w_1, \cdots, w_n, \alpha_i, \beta_i$. Hence, the map
\[ \tPhi \colon E(\NNN^k[G], v) \to \pi_1^{2k}[G,v], \quad \gamma \mapsto [\alpha_1 \cdot \beta_1 \cdots \alpha_n \cdot \beta_n]_{2k}\]
is well-defined.

Next, we show that $\tilde{\Phi}$ induces a map $\Phi \colon \mathcal{E}(\mathcal{N}^k[G],v) \to \pi_1^{2k}[G,v]$. Let $\gamma, \gamma' \in E(\NNN^k[G],v)$. Suppose that the pair $\gamma$ and $\gamma'$ of paths satisfies the condition (a) in Subsection~\ref{subsection edge-path}. Namely, we can write $\gamma = (v_0, \cdots, v_n)$ and $\gamma' = (v_0, \cdots, v_i, v_i, \cdots, v_n)$. Let $w_j \in V(G)$ such that $v_{j-1}, v_j \in N^k_G[w_j]$, $\alpha_j$ a path of length $k$ joining $v_{j-1}$ to $w_j$, and $\beta_j$ a path of length $k$ joining $w_j$ to $v_j$. Then, we have
\begin{eqnarray*}
\tilde{\Phi}(\gamma') &=& [\alpha_1 \cdot \beta_1 \cdots \alpha_i \cdot \beta_i \cdot \bar{\beta}_i \cdot \beta_i \cdot \alpha_{i+1} \cdots \beta_n]_{2k} = [\alpha_1 \cdot \beta_1 \cdots \alpha_n \cdot \beta_n]_{2k} = \tilde{\Phi}(\gamma).
\end{eqnarray*}

Next, suppose that the pair of paths $\gamma$ and $\gamma'$ satisfies the condition (b) in Subsection~\ref{subsection edge-path}. Then, we can write $\gamma = (v_0, \cdots, v_n)$ and $\gamma' = (v_0, \cdots, v_{i-1}, v_{i+1}, \cdots, v_n)$, where $\{ v_{i-1}, v_i, v_{i+1}\} \in \NNN^k[G]$. Let $w \in V(G)$ so that $v_{i-1}, v_i, v_{i+1} \in N^k_G[w]$, and $\alpha$ a path of length $k$ from $v_{i-1}$ to $w$, $\beta$ a path of length $k$ joining $w$ to $v_{i+1}$, and $\delta$ a path of length $k$ from $w$ to $v_i$. For $j \in \{ 1, \cdots, n\} \setminus \{ i, i +1\}$, let $w_j \in V(G)$ such that $v_{i-1}, v_i \in N^k_G[w_j]$, and let $\alpha_j$ be a path from $v_{j-1}$ to $w_j$ and $\beta_j$ be a path from $w_j$ to $v_j$. Then, we have
\begin{eqnarray*}
\tilde{\Phi}(\gamma) &=& [\alpha_1 \cdot \beta_1 \cdots \alpha_{i-1} \cdot \beta_{i-1} \cdot \alpha \cdot \delta \cdot \bar{\delta} \cdot \beta \cdot \alpha_{i+2} \cdot \beta_{i+2} \cdots \alpha_n \cdot \beta_n]_{2k} \\
&=& [\alpha_1 \cdot \beta_1 \cdots \alpha_{i-1} \cdot \beta_{i-1} \cdot \alpha \cdot \beta \cdot \alpha_{i+2} \cdot \beta_{i+2} \cdots \alpha_n \cdot \beta_n]_{2k}= \tilde{\Phi}(\gamma').
\end{eqnarray*}

Thus, we have completed the proof that $\gamma \simeq \gamma'$ implies $\tPhi(\gamma) = \tPhi(\gamma')$. Hence, $\tPhi$ induces a map $\Phi \colon \EEE(\NNN^k[G],v) \to \pi_1^{2k}[G,v]$. It immediately follows from the definition of $\Phi$ that it is a group homomorphism. Hence, to show that $\Phi$ is a group isomorphism, it suffices to prove the existence of the inverse $\Psi \colon \pi_1^{2k}[G,v] \to \mathcal{E}(\mathcal{N}^k[G],v)$.

Let $\gamma \colon P_n \to G$ be a closed path of $(G,v)$. We define $\hgamma \colon P_{2k \lceil n/2k \rceil} \to G$ as follows:
\[ \hgamma(i) = \begin{cases}
\gamma(i) & (i \le n) \\
v & (i \ge n).
\end{cases}\]
Here, $\lceil - \rceil$ denotes the ceiling function. Define $\hPsi \colon \Omega [G,v] \to E(\mathcal{N}^k[G],v)$ by
\[\hPsi(\gamma) = \left(\hgamma(0), \hgamma(2k), \dots, \hgamma \left(2k \left\lceil \frac{n}{2k} \right\rceil \right) \right).\]

We show that $\hPsi$ induces a map $\Psi \colon \pi_1^{2k}[G,v] \to \mathcal{E}(\mathcal{N}^k[G],v)$, i.e., $\gamma \simeq_{2k} \gamma'$ implies $\hPsi(\gamma) \simeq \hPsi(\gamma')$.

First, consider the case where the pair of paths $\gamma$ and $\gamma'$ satisfies the condition (A) in Subsection~\ref{subsection KY}. That is, if $l(\gamma) = n$, then $l(\gamma') = n+1$, and there is $x \in \{ 0, 1, \dots, n\}$ such that $\gamma(i) = \gamma'(i)$ for $i \le x$ and $\gamma(i) = \gamma'(i+1)$ for $i \ge x$. Let $m$ be an integer such that $2km \le x < 2k(m+1)$. In this case, we have
\[ \hPsi(\gamma) = \Big(\hgamma(0), \dots, \hgamma \left(2k \left\lceil \frac{n}{2k} \right\rceil \right) \Big),\ \hPsi(\gamma') = \Big(\hgamma'(0), \dots, \hgamma'\left(2k \left \lceil \frac{n+1}{2k} \right\rceil \right) \Big).\]
Suppose that $n$ is not a multiple of $2k$. Then, we have $2k \lceil n / 2k \rceil = 2k \lceil (n+1) / 2k \rceil$. Since
\[ \hgamma(2ki), \hgamma'(2k(i+1)), \hgamma(2k(i+1)) \in N^k_G[\hgamma(2ki + k)] \quad \text{and}\]
\[ \hgamma'(2ki), \hgamma'(2k(i+1)), \hgamma(2k(i+1)) \in N^k_G[\hgamma'(2ki + k)],\]
we have $\hPsi(\gamma) \simeq \hPsi(\gamma')$ (see Figure~\ref{figure homotopic}).

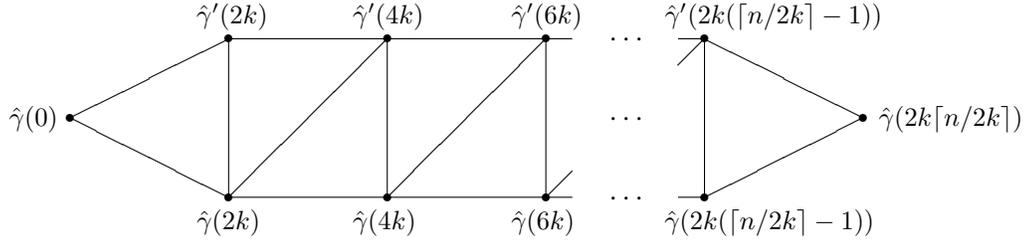
\begin{figure}
\centering
\begin{picture}(320,80)(0,-8)
\put(0,30){\circle*{3}}
\put(60,0){\circle*{3}}
\put(60,60){\circle*{3}}

\put(-23,27){\footnotesize $\hgamma(0)$}
\put(306,27){\footnotesize $\hgamma(2k \lceil n / 2k \rceil)$}

\put(48,-12){\footnotesize $\hgamma(2k)$}
\put(107,-12){\footnotesize $\hgamma(4k)$}
\put(167,-12){\footnotesize $\hgamma(6k)$}

\put(225,-12){\footnotesize $\hgamma(2k (\lceil n / 2k \rceil - 1))$}
\put(225,66){\footnotesize $\hgamma'(2k (\lceil n / 2k \rceil - 1))$}

\put(48,66){\footnotesize $\hgamma'(2k)$}
\put(107,66){\footnotesize $\hgamma'(4k)$}
\put(167,66){\footnotesize $\hgamma'(6k)$}

\put(0,30){\line(2,1){60}}
\put(0,30){\line(2,-1){60}}

\put(300,30){\line(-2,1){60}}
\put(300,30){\line(-2,-1){60}}

\put(60,0){\line(0,1){60}}
\put(120,0){\line(0,1){60}}
\put(180,0){\line(0,1){60}}
\put(240,0){\line(0,1){60}}

\put(60,60){\line(1,0){130}}
\put(60,0){\line(1,0){130}}
\put(180,0){\line(1,1){10}}

\put(204,-3){\small $\cdots$}
\put(204,27){\small $\cdots$}
\put(204,57){\small $\cdots$}

\put(240,0){\line(-1,0){10}}
\put(240,60){\line(-1,0){10}}
\put(240,60){\line(-1,-1){10}}

\put(60,0){\line(1,1){60}}
\put(120,0){\line(1,1){60}}

\put(120,0){\circle*{3}}
\put(120,60){\circle*{3}}

\put(180,0){\circle*{3}}
\put(180,60){\circle*{3}}

\put(240,0){\circle*{3}}
\put(240,60){\circle*{3}}

\put(300,30){\circle*{3}}
\end{picture}
\caption{$\hPsi(\gamma) \simeq \hPsi(\gamma')$}
\label{figure homotopic}
\end{figure}

Next, we consider the case where $n$ is a multiple of $2k$. Define $\delta \colon P_{n+1} \to G$ and $\delta' \colon P_{n+2} \to G$ to be the paths
\[ \delta(i) = \begin{cases}
\gamma(i) & i \le n \\
v & i = n+1
\end{cases} \quad \text{and} \quad \delta'(i) = \begin{cases}
\gamma'(i) & i \le n+1 \\
v & i = n+2.
\end{cases}\]
Since the pair $\delta$ and $\delta'$ satisfies (A) and $l(\delta)$ is not a multiple of $2k$, we have $\hPsi(\delta) \simeq \hPsi(\delta')$ by the previous paragraph. It is clear that $\hPsi(\gamma) \simeq \hPsi(\delta)$ and $\hPsi(\gamma') = \hPsi(\delta')$. Thus we have completed the proof of $\hPsi(\gamma) \simeq \hPsi(\gamma')$.

Thus, we have shown that if the pair $\gamma$ and $\gamma'$ satisfies the condition (A) in Subsection~\ref{subsection KY}, then $\hPsi(\gamma) \simeq \hPsi(\gamma')$. Next, we consider the case where the pair $\gamma$ and $\gamma'$ satisfies the condition (B)$_{2k}$. Namely, $l(\gamma) = l(\gamma')$ and there is $x \in \{ 0, 1, \dots, n\}$ such that $\gamma(i) = \gamma'(i)$ holds for $i \le x$ and $\gamma(i) = \gamma'(i)$ holds for $i \ge x + 2k$. Let $j$ be a non-negative integer such that $x + j$ is a multiple of $2k$. We define $\delta \colon P_{n+j} \to G$ and $\delta' \colon P_{n+j} \to G$ as:
\[ \delta(i) = \begin{cases}
v & (i \le j) \\
\gamma(i - j) & (i \ge j),
\end{cases}\quad \delta'(i) = \begin{cases}
v & (i \le j) \\
\gamma'(i - j) & (i \ge j)
\end{cases}\]
In this case, $\hPsi(\delta) = \hPsi(\delta')$. Since the homotopy class $\hPsi(\gamma)$ is invariant with respect to (A), we have
\[ \hPsi(\gamma) \simeq \hPsi(\delta) = \hPsi(\delta') \simeq \hPsi(\gamma'). \]

Hence, we have shown that the map $\hPsi \colon \Omega[G,v] \to E(\NNN^k[G],v)$ induces a map $\Psi \colon \pi_1^{2k}[G,v] \to \EEE(\NNN^k[G],v)$. By the definitions of $\Phi$ and $\Psi$, it is clear that $\Psi$ is the inverse of $\Phi$, and hence $\Phi$ is a group isomorphism. It is also clear that the isomorphism $\Phi \colon \pi_1(\NNN^k[G],v) \xrightarrow{\cong} \pi_1^{2k}[G,v]$ is natural with respect to graph maps. This completes the proof of Theorem~\ref{theorem B}.
\end{proof}

\appendix

\section{The case of directed graphs}

\begin{figure}
\centering
\begin{picture}(80,40)(0,0)
\put(0,20){\circle*{4}}
\put(40,20){\circle*{4}}
\put(80,20){\circle*{4}}
\put(0,20){\vector(1,0){23}}
\put(80,20){\vector(-1,0){23}}
\put(0,20){\line(1,0){40}}
\put(80,20){\line(-1,0){40}}
\end{picture}
\caption{Digraph $X_1$}
\label{figure non-isomorphic}
\end{figure}
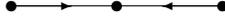

In this paper, we have mainly discussed graphs rather than digraphs. Here, we consider the analogues of the closed $k$-neighborhood complex for a digraph. We then show that the analogue of Theorem \ref{theorem B} does not hold for digraphs.


\begin{definition} \label{definition closed k}
Let $X$ be a digraph, $v$ a vertex of $X$ and $k$ a positive integer. Define the \emph{right closed $k$-neighborhood $\rN^k_X[v]$ of $v$ in $X$} recursively as follows:
\[ \rN^1_X[v] = \rN_X[v],\quad \rN^{k+1}_X[v] = \bigcup_{w \in \rN^k_X[v]} \rN_X[w] \quad (k \ge 1).\]
Similarly, define the \emph{left closed $k$-neighborhood $\lN^k_X[v]$ of $v$ in $X$} recursively as follows:
\[ \lN^1_X[v] = \lN_X[v],\quad \lN^{k+1}_X[v] = \bigcup_{w \in \lN^k_X[v]} \lN_X[w] \quad (k \ge 1).\]

We define the \emph{right closed $k$-neighborhood complex $\rNNN^k[X]$ of a digraph $X$} as follows: The ground set of $\rNNN^k[X]$ is the vertex set $V(X)$ of $X$. The set of simplices of $\rNNN^k[X]$ is the set of finite subsets of $V(X)$ contained in the right $k$-neighborhood of some vertex:
\[ \rNNN^k[X] = \{ \sigma \subset V(X) \mid \text{$\# \sigma < \infty$ and $\sigma \subset \rN^k_X[v]$ for some $v \in V(X)$}\}.\]
Similarly, the left $k$-neighborhood complex $\lNNN^k[X]$ of a digraph $X$ is defined as follows: The ground set of $\lNNN^k[X]$ is the vertex set $V(X)$ of $X$. The set of simplices of $\lNNN^k[X]$ is
\[ \lNNN^k[X] = \{ \sigma \subset V(X) \mid \text{$\# \sigma < \infty$ and $\sigma \subset \lN^k_X[v]$ for some $v \in V(X)$}\}.\]
\end{definition}

Note that $\rNNN^k[X]$ and $\lNNN^k[X]$ coincide with $\NNN^k[X]$ when $X$ is a graph.

In general, $\rNNN^k[X]$ and $\lNNN^k[X]$ are not homeomorphic. For example, consider the digraph $X_1$ depicted in Figure~\ref{figure non-isomorphic}. Then, $\rNNN^k[X_1]$ is homeomorphic to an interval, and $\lNNN^k[X_1]$ is homeomorphic to a $2$-simplex for every $k$. However, Dowker's classical theorem shows that these are homotopy equivalent:

\begin{theorem}[Dowker's theorem \cite{Dowker}] \label{theorem Dowker}
Let $X$ and $Y$ be sets and $S$ a subset of $X \times Y$. Define the simplicial complexes $K_L$ and $K_R$ as follows: The ground set of $K_L$ is $X$ and the set of simplices of $K_L$ is
\[ K_L = \{ \sigma \subset X \mid \text{$\# \sigma < \infty$ and there is $y \in Y$ such that $\sigma \times \{ y\} \subset S$}\}.\]
The ground set of $K_R$ is $Y$ and the set of simplices of $K_R$ is
\[ K_R = \{ \sigma \subset Y \mid \text{$\# \sigma < \infty$ and there is $x \in X$ such that $\{ x\} \times \sigma \subset S$}\}.\]
Then, $K_L$ and $K_R$ are homotopy equivalent.
\end{theorem}

\begin{figure}[t]
\centering
\begin{picture}(80,90)(0,-20)
\put(0,20){\circle*{4}}
\put(40,-20){\circle*{4}}
\put(40,60){\circle*{4}}
\put(80,20){\circle*{4}}

\put(0,20){\vector(1,1){22}}
\put(0,20){\line(1,1){40}}
\put(0,20){\vector(1,-1){22}}
\put(0,20){\line(1,-1){40}}
\put(80,20){\vector(-1,1){22}}
\put(80,20){\line(-1,1){40}}
\put(80,20){\vector(-1,-1){22}}
\put(80,20){\line(-1,-1){40}}
\end{picture}
\caption{Digraph $X_2$}
\label{figure counter example}
\end{figure}
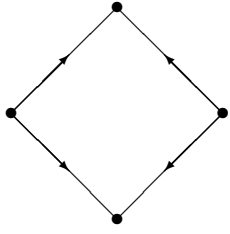

\begin{corollary}
Let $X$ be a digraph and $k$ a positive integer. Then, $\rNNN^k[X]$ and $\lNNN^k[X]$ are homotopy equivalent.
\end{corollary}
\begin{proof}
Set
\[ S = \{ (v,w) \in V(X) \times V(X) \mid w \in \rN^k_X[v]\} \subset V(X) \times V(X).\]
Then, $K_L$ coincides with $\lNNN^k[X]$, and $K_R$ coincides with $\rNNN^k[X]$. Thus, Theorem~\ref{theorem Dowker} implies $\rNNN^k[X] \simeq \lNNN^k[X]$.
\end{proof}

It is easy to see that the generalization of Theorem~\ref{theorem B} does not hold for digraphs. Namely, for every $k \ge 2$ there is a digraph $X$ such that $\pi_1^{2k}[X]$ and $\pi_1(\rNNN^k[X])$ ($\cong \pi_1(\lNNN^k[X])$) are not isomorphic. Such an example is given by the graph $X_2$ described in \ref{figure counter example}. Indeed, for $k \ge 1$ the right closed $k$-neighborhood complex $\rNNN^k[X_2]$ and left closed $k$-neighborhood complex $\lNNN^k[X_2]$ of $X_2$ are homeomorphic to a square, and hence $\pi_1(\rNNN^k[X_2])$ and $\pi_1(\lNNN^k[X_2])$ are trivial. On the other hand, $\pi_1^k[X_2]$ is isomorphic to $\ZZ$. To see this, let $\tX_2$ be the graph depicted in Figure~\ref{figure universal covering}. Namely, $V(\tX_2) = \ZZ$ and $E(\tX_2) = \{ (2i, 2i\pm 1) \mid i \in \ZZ\}$. Then, the natural covering map $\tX_2 \to X_2$ is indeed $k$-covering for every $k \ge 1$. It is clear that $\pi_1(\tX_2)$ is trivial. Hence, by the covering theory developed in \cite{DIMZ}, we conclude that $\pi_1[X_2] \cong \ZZ$.

\begin{figure}[t]
\centering
\begin{picture}(0,40)

\put(-140,20){\line(1,0){18}}
\put(-160,17){$\cdots$}

\put(-80,20){\circle*{4}}
\put(-80,20){\vector(1,0){23}}
\put(-80,20){\line(1,0){40}}
\put(-40,20){\circle*{4}}

\put(-80,20){\vector(-1,0){23}}
\put(-80,20){\line(-1,0){40}}
\put(-120,20){\circle*{4}}

\put(0,20){\circle*{4}}
\put(0,20){\vector(1,0){23}}
\put(0,20){\line(1,0){40}}
\put(40,20){\circle*{4}}

\put(0,20){\vector(-1,0){23}}
\put(0,20){\line(-1,0){40}}
\put(-40,20){\circle*{4}}

\put(80,20){\circle*{4}}
\put(80,20){\vector(1,0){23}}
\put(80,20){\line(1,0){40}}
\put(120,20){\circle*{4}}

\put(80,20){\vector(-1,0){23}}
\put(80,20){\line(-1,0){40}}
\put(40,20){\circle*{4}}

\put(140,20){\line(-1,0){18}}
\put(146,17){$\cdots$}
\end{picture}
\caption{Digraph $\tX_2$}
\label{figure universal covering}
\end{figure}

\bibliographystyle{abbrvurl} %
\bibliography{reference} %

\end{document}